\documentclass{article}
\usepackage{amsmath}
\usepackage{amssymb}
\usepackage{amsthm}
\newtheorem{theorem}{Theorem}[section]
\newtheorem{lemma}[theorem]{Lemma}
\newtheorem{corollary}[theorem]{Corollary}
\newtheorem{proposition}[theorem]{Proposition}

\newtheorem*{thma}{Theorem A}

\newtheorem*{thmdvo}{Theorem D}

\numberwithin{equation}{section}

\long\def\symbolfootnote[#1]#2{\begingroup%
\def\thefootnote{\fnsymbol{footnote}}\footnote[#1]{#2}\endgroup}

\begin{document}

\def\C{{\mathbb C}}
\def\N{{\mathbb N}}
\def\Z{{\mathbb Z}}
\def\R{{\mathbb R}}
\def\K{{\mathbb K}}
\def\E{{\cal E}}
\def\epsilon{\varepsilon}
\def\kappa{\varkappa}
\def\phi{\varphi}
\def\leq{\leqslant}
\def\geq{\geqslant}
\def\re{\text{\tt Re}\,}
\def\ilim{\mathop{\hbox{$\underline{\hbox{\rm lim}}$}}\limits}
\def\dim{\hbox{\tt dim}\,}
\def\ker{\hbox{\tt ker}\,}
\def\spann{\hbox{\tt span}\,}
\def\Re{\hbox{\tt Re}\,}
\def\ssub#1#2{#1_{{}_{{\scriptstyle #2}}}}

\title{The Kitai Criterion and backward shifts}

\author{Stanislav Shkarin}

\date{}

\maketitle

\begin{abstract} \noindent It is proved that for any separable infinite
dimensional Banach space $X$, there is a bounded linear operator $T$
on $X$ such that $T$ satisfies the Kitai Criterion. The proof is
based on quasisimilarity argument and on showing that $I+T$
satisfies the Kitai Criterion for certain backward weighted shifts
$T$.
\end{abstract}



\section{Introduction \label{s1}}\rm

All vector spaces are assumed to be over $\K$ being either the field
$\C$ of complex numbers or the field $\R$ of real numbers.
As\symbolfootnote[0]{Partially supported by Plan Nacional I+D+I
grant no. MTM2006-09060, Junta de Andaluc\'{\i}a FQM-260 and British
Engineering and Physical Research Council Grant GR/T25552/01.}
usual, $\Z$ is the set of integers, $\Z_+$ is the set of
non-negative integers and $\N$ is the set of positive integers. For
a Banach space $X$, symbol $L(X)$ stands for the space of bounded
linear operators on $X$ and $X^*$ is the space of continuous linear
functionals on $X$.

\medskip

\noindent{\bf Definition 1.} \ Let $X$ be a Banach space and $T\in
L(X)$. We say that $T$ is {\it hypercyclic} if there exists $x\in X$
such that the orbit $\{T^nx:n\in \Z_+\}$ is dense in $X$. Following
Ansari \cite{ansa}, $T$ is called {\it hereditarily hypercyclic} if
for any infinite set $\Lambda\subseteq \Z_+$, there exists $x\in X$
for which $\{T^nx:n\in\Lambda\}$ is dense in $X$.

\medskip

\noindent{\bf Remark.} \ Hypercyclic operators have been intensely
studied during last few decades, see surveys \cite{ge1,ge2} and
references therein. It is also worth noting that in the terminology
of \cite{bp,gri} hereditarily hypercyclic operators are called
hereditarily hypercyclic with respect to the sequence $n_k=k$ of all
non-negative integers.

\medskip

\noindent{\bf Definition 2.} \ We say that a bounded linear operator
$T$ on a Banach space $X$ satisfies the {\it Kitai Criterion}
\cite{kitai} if there exist two dense subsets $E$ and $F$ of $X$ and
a map $S:F\to F$ such that $TSy=y$, $S^k y\to 0$ and $T^{k}x\to 0$
as $k\to\infty$ for any $y\in F$ and $x\in E$.

\medskip

\noindent{\bf Definition 3.} \ Let $X$ be a topological space. A
continuous map $T:X\to X$ is called {\it mixing} if for any two
non-empty open sets $U,V\subseteq X$, $T^n(U)\cap V\neq \varnothing$
for all sufficiently large $n\in\N$.

\medskip

It is well-known that a bounded linear operator $T$ on a Banach
space is mixing if and only if it is hereditarily hypercyclic, see
\cite{gri}. Moreover, if $T$ satisfies the Kitai Criterion, then $T$
is mixing, see \cite{gri,CS}. Grivaux \cite{gri}, answering a
question raised by Shapiro, proved that for any separable infinite
dimensional Banach space $X$, there exists a mixing $T\in L(X)$,
that does not satisfy the Kitai criterion. We address the question
raised by Petersson (Miniworkshop: Hypercyclicity and Linear Chaos,
Oberwolfach, August, 2006), whether there is a bounded linear
operator satisfying the Kitai Criterion on any separable infinite
dimensional Banach space. The following theorem gives an affirmative
answer to this question.

\begin{theorem}\label{kitai}
Let $X$ be a separable infinite dimensional Banach space. Then there
exists $T\in L(X)$ satisfying the Kitai Criterion.
\end{theorem}\rm

This theorem fits into the following chain of results. Gerzog
\cite{ger} proved that there is a supercyclic continuous linear
operator on any separable infinite dimensional Banach space. Later
Ansari \cite{ansa1} and Bernal-Gonz\'ales \cite{bernal} showed
independently that for any separable infinite dimensional Banach
space $X$ there is a hypercyclic operator $T\in L(X)$. Finally, as
we have already mentioned, Grivaux \cite{gri} proved that there is a
mixing operator on any separable infinite dimensional Banach space.

Recall that a {\it backward weighted shift} on $\ell_p=\ell_p(\Z_+)$
for $1\leq p<\infty$ or $c_0=c_0(\Z_+)$ is the operator $T$ acting
on the canonical basis $\{e_n\}_{n=0}^\infty$ as follows: $Te_0=0$
and $Te_n=w_ne_{n-1}$ for $n\geq 1$, where $w=\{w_n\}_{n\in\N}$ is a
bounded sequence of non-zero numbers in $\K$. The proof of
Theorem~\ref{kitai} is based on the fact that operators $I+T$
satisfy the Kitai Criterion for certain backward weighted shifts
$T$. Namely, the following theorem holds.

\begin{theorem}\label{bws1} Let $X$ be either $\ell_p$ with
$1\leq p<\infty$ or $X=c_0$ and $T:X\to X$ be a backward weighted
shift with the weight sequence $\{w_n\}_{n\in\N}$. Assume also that
\begin{equation}\label{1/n}
\ilim_{n\to\infty} \biggl(n!\prod_{j=1}^n |w_j|\biggr)^{1/n}>0.
\end{equation}
Then $I+T$ satisfies the Kitai criterion. \end{theorem}

It is worth noting that if there exists $c>0$ such that $|w_n|\geq
c/n$ for each $n\in\N$, then (\ref{1/n}) is satisfied. Thus, we have
the following corollary.

\begin{corollary}\label{bws2} Let $X$ be either $\ell_p$ with
$1\leq p<\infty$ or $X=c_0$ and $T:X\to X$ be a backward weighted
shift with the weight sequence $\{w_n\}_{n\in\N}$. Assume also that
there exists $c>0$ such that $|w_n|\geq c/n$ for each $n\in\N$. Then
$I+T$ satisfies the Kitai criterion. \end{corollary}

Recall that according to a theorem by Salas \cite{sal} the operator
$I+T$ is hypercyclic for any backward weighted shift $T$ on $\ell_p$
with $1\leq p<\infty$. Grivaux \cite{gri} has observed that
basically the same proof allows to show that the operators $I+T$ in
the Salas theorem are mixing. Theorem~\ref{bws1} shows that under
certain conditions on the weight sequence these operators satisfy
the Kitai criterion.

Theorems~\ref{kitai} and~\ref{bws1} are proved in Sections~\ref{s2}.
Section~\ref{s3} is devoted to concluding remarks.

\section{Operators satisfying the Kitai Criterion \label{s2}}

In this section we shall prove Theorems~\ref{kitai} and~\ref{bws1}.
We need some preparation.

\subsection{Biorthogonal sequences \label{s2-1}}\rm

We are going to prove the existence of biorthogonal sequences with
certain properties on any separable infinite dimensional Banach
space. Recall that a family $\{(x_\alpha,f_\alpha):\alpha\in A\}$,
where $x_\alpha$ are vectors in a Banach space $X$ and $f_\alpha\in
X^*$, is called {\it biorthogonal} if
$f_\alpha(x_\beta)=\delta_{\alpha,\beta}$ for any $\alpha,\beta\in
A$, $\delta_{\alpha,\beta}$ being the Kronecker delta, that is
$\delta_{\alpha,\beta}=1$ if $\alpha=\beta$ and
$\delta_{\alpha,\beta}=0$ if $\alpha\neq\beta$.

\begin{proposition}\label{biort} Let $X$ be a separable infinite
dimensional Banach space, $\{b_k\}_{k\in\Z_+}$ be sequence of
numbers in $[3,\infty)$ such that $b_k\to\infty$ as $k\to\infty$ and
$\{n_k\}_{k\in\Z_+}$ be a strictly increasing sequence of positive
integers such that $n_0=0$ and $n_{k+1}-n_k\geq 2$ for each
$k\in\Z_+$. Then there exists a biorthogonal sequence
$\{(x_k,f_k)\}_{k\in\Z_+}$ in $X\times X^*$ such that
\begin{itemize}
\item[\rm(B1)]$\|x_k\|=1$ for each $k\in\Z_+;$
\item[\rm(B2)]$\spann\{x_k:k\in\Z_+\}$ is dense in $X;$
\item[\rm(B3)]$\|f_{n_k}\|\leq b_k$ for each $k\in\Z_+;$
\item[\rm(B4)]$\|f_j\|\leq 3$ if
$j\in\Z_+\setminus\{n_k:k\in\Z_+\};$
\item[\rm(B5)]for any $k\in\Z_+$ and any
numbers $c_{j}\in\K$ with $n_k+1\leq j\leq n_{k+1}-1,$
$$
\frac12\left\|\sum_{j=n_k+1}^{n_{k+1}-1}c_jx_j\right\|\leq
\left(\sum_{j=n_k+1}^{n_{k+1}-1}|c_j|^2\right)^{1/2}\!\!\!\!\!\leq
2\left\|\sum_{j=n_k+1}^{n_{k+1}-1}c_jx_j\right\|.
$$
\end{itemize}
\end{proposition}

The proof of the above proposition is based on the following
theorem, which is known as Dvoretzky theorem on almost spherical
sections \cite{Dv}.

\begin{thmdvo} For each $n\in\N$ and each $\epsilon>0$,
there exists $m=m(n,\epsilon)\in\N$ such that for any Banach space
$X$ with $\dim X\geq m$ there is an $n$-dimensional linear subspace
$L$ in $X$ and a linear basis $e_1,\dots,e_n$ in $L$ for which
$\|e_j\|=1$ for $1\leq j\leq n$ and
$$
\frac1{1+\epsilon}\biggl\|\sum_{j=1}^nc_je_j\biggr\|\leq
\biggl(\sum_{j=1}^n|c_j|^2\biggr)^{1/2}\!\!\!\!\!\leq (1+\epsilon)
\biggr\|\sum_{j=1}^nc_je_j\biggr\|\ \ \text{for any
$(c_1,\dots,c_n)\in\K^n$}.
$$
\end{thmdvo}

We also need the following technical lemmas.

\begin{lemma}\label{seq}
Let $A$ be an infinite countable set and $\{a_n\}_{n\in A}$ be a
sequence of non-negative numbers such that there exists $c>0$ for
which the set $\{n\in A:a_n\leq c\}$ is infinite and
$\{b_n\}_{n\in\Z_+}$ be a sequence of numbers in $[c,\infty)$ such
that $b_n\to\infty$ as $n\to\infty$. Then there exists a bijection
$\pi:\Z_+\to A$ such that $a_{\pi(n)}\leq b_n$ for each $n\in\Z_+$.
\end{lemma}

\begin{proof} Since the set $\{n\in A:a_n\leq c\}$
is infinite, we can choose two disjoint infinite subsets $B$ and $C$
of $A$ such that $a_{n}\leq c$ for each $n\in B$ and $B\cup C=A$.
Fix a bijection $\phi:\Z_+\to C$. Since $b_n\to\infty$ as
$n\to\infty$, we can choose a sequence $\{j_k\}_{k\in\Z_+}$ of
non-negative integers such that $j_{k+1}-j_k\geq 2$ and
$a_{\phi(k)}\leq b_{j_k}$ for each $k\in\Z_+$. Denote
$D=\Z_+\setminus\{j_k:k\in\Z_+\}$. Since $j_{k+1}-j_k\geq 2$, the
set $D$ is infinite. Since both $B$ and $D$ are infinite and
countable, there exists a bijection $\psi:D\to B$. Now consider the
map $\pi:\Z_+\to A$ defined by the formula
$$
\pi(n)=\left\{\begin{array}{ll}\phi(k)&\text{if $n=j_k$},\\
\psi(n)&\text{if $n\in D$}.\end{array}\right.
$$
Clearly $\pi$ is a bijection. If $n\in D$, we have
$a_{\pi(n)}=a_{\psi(n)}\leq c\leq b_{n}$. Indeed, $a_m\leq c$ for
$m\in B$ and $c\leq b_n$ for $n\in\Z_+$. If $n=j_k$, we have
$a_{\pi(n)}=a_{\phi(k)}\leq b_{j_k}=b_{n}$. Thus, $a_{\pi(n)}\leq
b_n$ for each $n\in\Z_+$.
\end{proof}

\begin{lemma}\label{bi1}
Let $X$ be an infinite dimensional Banach space, $u\in X$ and
$\{(x_k,f_k):1\leq k\leq n\}$ be a finite biorthogonal sequence in
$X\times X^*$. Then there exist $x_{n+1}\in X$ and $f_{n+1}\in X^*$
such that $\|x_{n+1}\|=1$, $\{(x_k,f_k):1\leq k\leq n+1\}$ is
biorthogonal and $u\in\spann\{x_1,\dots,x_{n+1}\}$.
\end{lemma}

\begin{proof} Since $X=L\oplus N$, where $L=\spann\{x_1,\dots,x_n\}$
and $N=\bigcap\limits_{k=1}^n\ker f_j$, we can pick a vector
$x_{n+1}\in N$ such that $\|x_{n+1}\|=1$ and $u\in
M=\spann(L\cup\{x_{n+1}\})=\spann\{x_1,\dots,x_{n+1}\}$. Since
$x_1,\dots,x_{n+1}$ is a linear basis in $M$, there exists a unique
linear functional $g:M\to \K$ such that $g(x_j)=0$ for $1\leq j\leq
n$ and $g(x_{n+1})=1$. According to the Hahn--Banach theorem, there
exists $f_{n+1}\in X^*$ such that $f_{n+1}\bigr|_{M}=g$. Clearly the
pair $(x_{n+1},f_{n+1})$ satisfies all required conditions.
\end{proof}

\begin{lemma}\label{bi2}
Let $m\in\N$, $X$ be an infinite dimensional Banach space and
$\{(x_k,f_k):1\leq k\leq n\}$ be a finite biorthogonal sequence in
$X\times X^*$. Then there exist $x_{n+1},\dots,x_{n+m}\in X$ and
$f_{n+1},\dots,f_{n+m}\in X^*$ such that $\{(x_k,f_k):1\leq k\leq
n+m\}$ is biorthogonal, $\|x_k\|=1$ and $\|f_k\|\leq 3$ for $n+1\leq
k\leq n+m$ and
\begin{equation}\label{dr}
\frac12\biggl\|\sum_{j=n+1}^{n+m}c_jx_j\biggr\|\leq
\biggl(\sum_{j=n+1}^{n+m}|c_j|^2\biggr)^{1/2}\!\!\!\!\!\leq
2\biggl\|\sum_{j=n+1}^{n+m}c_jx_j\biggr\|\ \ \text{for any
$(c_{n+1},\dots,c_{n+m})\in\K^m$}.
\end{equation}\end{lemma}

\begin{proof} Fix $\delta\in(0,1/2)$ such that
$\frac{(1+\delta)(2-\delta)}{1-\delta}<3$ and let
$L=\spann\{x_1,\dots,x_n\}$, $S=\{x\in L:\|x\|=1\}$. Since $L$ is
finite-dimensional, $S$ is a compact metric space with respect to
the metric inherited from $X$. Hence we can find a finite set
$\{u_1,\dots,u_r\}\subset S$ such that for each $x\in S$ there
exists $j\in\{1,\dots,r\}$ for which $\|x-u_j\|<\delta$. According
to the Hahn--Banach theorem, there exist functionals
$\phi_1,\dots,\phi_r\in X^*$ such that $\|\phi_j\|=\phi_j(u_j)=1$
for $1\leq j\leq r$. Consider the space
$$
M=\bigcap_{l=1}^n\ker f_j\ \cap\ \bigcap_{j=1}^r \ker \phi_j.
$$
Since $X$ is infinite dimensional, $M$ is a closed infinite
dimensional subspace of $X$. By Theorem~D there exist
$x_{n+1},\dots,x_{n+m}\in M$ such that $\|x_j\|=1$ for $n+1\leq
j\leq n+m$ and
\begin{equation}\label{dr1}
\frac1{1+\delta}\biggl\|\sum_{j=n+1}^{n+m}c_jx_j\biggr\|\leq
\biggl(\sum_{j=n+1}^{n+m}|c_j|^2\biggr)^{1/2}\!\!\!\!\!\leq
(1+\delta) \biggl\|\sum_{j=n+1}^{n+m}c_jx_j\biggr\|\ \ \text{for any
$(c_{n+1},\dots,c_{n+m})\in\K^m$}.
\end{equation}
Since $\delta\leq1/2$, (\ref{dr1}) implies (\ref{dr}).

Let now $L_0=\spann\{x_{n+1},\dots,x_{n+m}\}$. Since $L_0\subset M$
and the linear functionals $f_j$ for $1\leq j\leq n$ vanish on $M$,
we see that vectors $x_1,\dots,x_{n+m}$ form a linear basis in
$L\oplus L_0$. Thus, there exist unique linear functionals
$g_{n+1},\dots,g_{n+m}$ on $L\oplus L_0$ such that
$g_j(x_k)=\delta_{j,k}$ for $n+1\leq j\leq n+m$ and $1\leq k\leq
n+m$. We shall estimate the norms of $g_j$.

First, let $y=\sum_{j=n+1}^{n+m}c_jx_j\in L_0$ and $n+1\leq k\leq
n+m$. Using (\ref{dr1}), we obtain
$$
|g_k(y)|=|c_k|\leq \biggl(\sum_{j=n+1}^{n+m}|c_j|^2\biggr)^{1/2}
\!\!\!\!\!\leq (1+\delta)\|y\|.
$$
Thus,
\begin{equation}\label{gk}
|g_k(y)|\leq (1+\delta)\|y\|\quad \text{for each $y\in L_0$.}
\end{equation}
Let now $x\in L$ with $\|x\|=1$ and $y\in L_0$. Pick
$j\in\{1,\dots,r\}$ such that $\|x-u_j\|\leq \delta$. Then
$$
\|x+y\|\geq |\phi_j(x+y)|=|\phi_j(x)|\geq
|\phi_j(u_j)|-|\phi_j(x-u_j)|\geq 1-\delta.
$$
It follows that $\|x+y\|\geq (1-\delta)\|x\|$ for each $x\in L$ and
$y\in L_0$. Consequently
\begin{equation}\label{gk1}
\|y\|\leq \|x\|+\|x+y\|\leq\biggl(1+\frac1{1-\delta}\biggr)\|x+y\|=
\frac{2-\delta}{1-\delta}\|x+y\| \quad \text{for each $x\in L$ and
$y\in L_0$.}
\end{equation}
Let now $w\in L\oplus L_0$. Then there exist unique $x\in L$ and
$y\in L_0$ for which $w=x+y$. Applying (\ref{gk}) and (\ref{gk1}),
we obtain
$$
|g_k(w)|=|g_k(x+y)|=|g_k(y)|\leq (1+\delta)\|y\|\leq
\frac{(1+\delta)(2-\delta)}{1-\delta}\|w\|.
$$
That is, $\|g_k\|\leq \frac{(1+\delta)(2-\delta)}{1-\delta}\leq 3$
for $n+1\leq k\leq n+m$. Using the Hahn--Banach theorem, we can now
choose $f_k\in X^*$ for $n+1\leq k\leq n+m$ such that
$f_k\bigr|_{L\oplus L_0}=g_k$ and $\|f_k\|\leq 3$. It remains to
observe that $\{(x_k,f_k):1\leq k\leq n+m\}$ is a biorthogonal
sequence satisfying all desired conditions. \end{proof}

\begin{proof}[Proof of Proposition~\ref{biort}]
Since $X$ is separable, we can choose a sequence
$\{u_k\}_{k\in\Z_+}$ of vectors in $X$ such that
$\spann\{u_k:k\in\Z_+\}$ is dense in $X$. Denote $m_k=n_k+k$ for
$k\in\Z_+$. First, according to the Hahn--Banach theorem, we can
pick $(y_0,g_0)\in X\times X^*$ such that $\|y_0\|=\|g_0\|=1$ and
$u_0\in\spann\{y_0\}$, which will serve as the basis of an inductive
process. Now we shall show inductively that for $k=0,1,\dots$, there
exist $\{(y_j,g_j):m_k<j\leq m_{k+1}\}\subset X\times X^*$ such that
\begin{itemize}
\item[(P1)]$\{(y_j,g_j):0\leq j\leq m_{k+1}\}$ is a biorthogonal
sequence;
\item[(P2)]$\|y_j\|=1$ for $m_k<j\leq m_{k+1}$;
\item[(P3)]$u_{k+1}\in\spann\{y_0,\dots,y_{m_{k+1}}\}$;
\item[(P4)]$\|g_j\|\leq 3$ for $m_k<j<m_{k+1}$;
\item[(P5)]for any
$(c_{m_k+1},\dots,c_{m_{k+1}-1})\in\K^{m_{k+1}-m_k-1}$,
$$
\frac12\left\|\sum_{j=m_k+1}^{m_{k+1}-1}c_jy_j\right\|\leq
\left(\sum_{j=m_k+1}^{m_{k+1}-1}|c_j|^2\right)^{1/2}\!\!\!\!\!\leq
2\left\|\sum_{j=m_k+1}^{m_{k+1}-1}c_jy_j\right\|.
$$
\end{itemize}
Suppose $(y_j,g_j)$ for $j\leq m_k$ satisfying the desired
properties are already constructed. Using Lemma~\ref{bi2}, we find
$(y_j,g_j)$ for $m_k<j<m_{k+1}$ in $X\times X^*$ such that
$\{(y_j,g_j):0\leq j<m_{k+1}\}$ is a biorthogonal sequence,
$\|y_j\|=1$ for $m_k<j<m_{k+1}$ and the conditions (P4) and (P5) are
satisfied. By Lemma~\ref{bi1}, we can now choose
$(y_{m_{k+1}},g_{m_{k+1}})\in X\times X^*$ such that
$\|y_{m_{k+1}}\|=1$ and the conditions (P1) and (P3) are satisfied.
Therefore $\{(y_j,g_j):m_k<j\leq m_{k+1}\}$ satisfy (P1--P5).

Thus we have just defined an inductive procedure, which provides a
sequence $\{(y_j,g_j):j\in\Z_+\}$ of elements of $X\times X^*$
satisfying
\begin{itemize}
\item[(Q1)]$\{(y_j,g_j):j\in\Z_+\}$ is a biorthogonal
sequence;
\item[(Q2)]$\|y_j\|=1$ for $j\in\Z_+$;
\item[(Q3)]$\spann\{u_j:j\in\Z_+\}\subseteq \spann\{y_j:j\in\Z_+\}$;
\item[(Q4)]$\|g_j\|\leq 3$ for $m_k<j<m_{k+1}$, $k\in\Z_+$;
\item[(Q5)]for any $k\in\Z_+$ and
$(c_{m_k+1},\dots,c_{m_{k+1}-1})\in\K^{m_{k+1}-m_k-1}$,
$$
\frac12\left\|\sum_{j=m_k+1}^{m_{k+1}-1}c_jy_j\right\|\leq
\left(\sum_{j=m_k+1}^{m_{k+1}-1}|c_j|^2\right)^{1/2}\!\!\!\!\!\leq
2\left\|\sum_{j=m_k+1}^{m_{k+1}-1}c_jy_j\right\|.
$$
\end{itemize}

Now we shall show that a biorthogonal sequence with the desired
properties can be obtained as a permutation of the biorthogonal
sequence $\{(y_j,g_j)\}$. Let
$$
A=B\cup C,\quad \text{where}\quad B=\{m_k:k\in\Z_+\}\quad
\text{and}\quad C=\{m_k+1:k\in\Z_+\}.
$$
Clearly $A$ is an infinite countable set. According to (Q4)
$\|g_j\|\leq 3$ if $j\in B$. By Lemma~\ref{seq}, there exists a
bijection $\phi:\Z_+\to A$ such that $\|g_{\phi(j)}\|\leq b_j$ for
each $j\in\Z_+$. Now we consider the map $\pi:\Z_+\to\Z_+$ defined
by the formula
$$
\pi(j)=\left\{\begin{array}{ll} \phi(k)&\text{if $j=n_k$},\\
m_k+l+1&\text{if $j=n_k+l$, $1\leq
l<n_{k+1}-n_k$.}\end{array}\right.
$$
Using the equalities $m_k=n_k+k$ and $m_{k+1}-m_k=n_{k+1}-n_k+1$, it
is easy to verify that $\pi$ maps bijectively
$\Z_+\setminus\{n_k:k\in\Z_+\}$ onto $\Z_+\setminus A$. Since
$\phi:\Z_+\to A$ is also bijective, we have that $\pi:\Z_+\to\Z_+$
is a bijection. Thus, $\{(x_j,f_j):j\in\Z_+\}$ with
$(x_j,f_j)=(y_{\pi(j)},g_{\pi(j)})$ is just a permutation of the
biorthogonal sequence $\{(y_j,g_j):j\in\Z_+\}$. Hence
$\{(x_j,f_j):j\in\Z_+\}$ is a also a biorthogonal sequence and the
properties (B1) and (B2) follow immediately from (Q2) and (Q3).
Since $\pi(\{n_k+1,\dots,n_{k+1}-1\})=\{m_k+2,\dots,m_{k+1}-1\}$ for
each $k\in\Z_+$, properties (B4) and (B5) follow from (Q4) and (Q5)
respectively. Finally $\|f_{n_k}\|=\|g_{\phi(k)}\|\leq b_k$ for each
$k\in\Z_+$ and (B3) is also satisfied.
\end{proof}

We need Proposition~\ref{biort} for one purpose only. Namely, we
need it in order to prove the following lemma which is one of the
two main ingredients in the proof of Theorem~\ref{kitai}.

\begin{lemma}\label{wshi} Let $a\in\ell_2(\N)$ and $T_a$ be the
weighted backward shift on $\ell_1$ with the weight sequence $a$,
that is, $T_ae_0=0$ and $T_ae_n=a_ne_{n-1}$ for $n\geq 1$, where
$\{e_n\}_{n\in\Z_+}$ is the canonical basis in $\ell_1$. Let also
$X$ be a separable infinite dimensional Banach space. Then there
exist $T\in L(X)$ and an injective bounded linear operator
$J:\ell_1\to X$ with dense range such that $JT_a=TJ$.
\end{lemma}

\begin{proof} Since $a\in\ell_2(\N)$, we can choose
a sequence $\{n_k:k\in\Z_+\}$ of non-negative integers and a
sequence $\{b_k:k\in\Z_+\}$ of numbers in $[3,\infty)$ such that
$n_0=0$, $n_{k+1}-n_k\geq 2$ for each $k\in\Z_+$, $b_k\to\infty$ as
$k\to\infty$,
\begin{equation}
\sum_{k=0}^\infty b_k|a_{n_k}|<\infty\quad\text{and}\quad
\sum_{k=0}^\infty
\left(\sum_{j=n_k+1}^{n_{k+1}}|a_j|^2\right)^{1/2}<\infty.
\label{su2}
\end{equation}
Let $\{(x_j,f_j):j\in\Z_+\}$ be the biorthogonal system in $X\times
X^*$ furnished by Proposition~\ref{biort} with the above choice of
sequences $\{n_k\}$ and $\{b_k\}$. Since $\|x_n\|=1$ for each
$n\in\Z_+$, the formula
$$
Ju=\sum_{n=0}^\infty u_n x_n
$$
defines a bounded linear operator from $\ell_1$ to $X$. Since
$f_j(Ju)=u_j$ for any $j\in\Z_+$ and $u\in\ell_1$, we see that $J$
is injective. Since $J(e_j)=x_j$ for each $j\in\Z_+$ and
$\spann\{x_j:j\in\Z_+\}$ is dense in $X$, the operator $J$ has dense
range. For each $k\in\Z_+$, we consider $S_k\in L(X)$ defined by the
formula
$$
S_kx=\sum_{j=n_k+2}^{n_{k+1}-1} a_jf_j(x)x_{j-1},
$$
assuming that $S_k=0$ if $n_k-n_{k-1}=2$. Applying
Proposition~\ref{biort} (B5) and (B4), we have
$$
\|S_kx\|^2\leq 4\sum_{j=n_k+2}^{n_{k+1}-1} |a_j|^2|f_j(x)|^2\leq
36\|x\|^2\sum_{j=n_k+2}^{n_{k+1}-1} |a_j|^2\leq
36\|x\|^2\sum_{j=n_k+1}^{n_{k+1}}|a_j|^2.
$$
Hence
$$
\|S_k\|\leq 6\left(\sum_{j=n_k+1}^{n_{k+1}}|a_j|^2\right)^{1/2}.
$$
Thus applying the second inequality in (\ref{su2}), we see that
$\sum_{k=0}^\infty\|S_k\|<\infty$ and therefore the formula
$$
T_0x=\sum_{k=0}^\infty S_kx
$$
defines a bounded linear operator on $X$. Since by
Proposition~\ref{biort} (B3), $\|f_{n_k}\|\leq b_k$, we, using the
first inequality in (\ref{su2}), see that the formula
$$
T_1x=\sum_{k=1}^\infty a_{n_k}f_{n_k}(x) x_{n_k-1}
$$
also defines a bounded linear operator on $X$. Finally, the second
inequality in (\ref{su2}) implies that
$\sum_{k=0}^\infty|a_{n_k+1}|<\infty$. Thus, since
$\|f_{n_k+1}\|\leq 3$ according to Proposition~\ref{biort} (B4), the
formula
$$
T_2x=\sum_{k=0}^\infty a_{n_k+1}f_{n_k+1}(x) x_{n_k}
$$
defines a bounded linear operator on $X$ as well. Consider now the
operator $T=T_0+T_1+T_2\in L(X)$. It is straightforward to verify
that $Tx_0=0$ and $Tx_n=a_nx_{n-1}$ for $n\geq 1$. Taking into
account that $Je_n=x_n$ for each $n\in\Z_+$, we see that
$JT_ae_0=TJe_0=0$ and $JT_ae_n=TJe_n=a_nx_{n-1}$ for $n\geq 1$.
Hence $JT_a=TJ$.
\end{proof}

\subsection{Orbits of $I-cV$ and $(I+cV)^{-1}$ \label{s2-2}}

Let $V$ be the classical Volterra operator acting on $L_2[0,1]$:
$$
Vf(t)=\int_0^t f(s)\,ds.
$$
We need the following fact on the behavior of the orbits of $I-cV$
and $(I+cV)^{-1}$ for $c>0$. For the case $c=1$, the proof can be
found in \cite{msz} and the proof for general $c>0$ goes along the
same lines. We include the proof for sake of completeness.

\begin{lemma}\label{volterra}
Let $c>0$. Then for any $f\in L_2[0,1]$,
\begin{align}\label{volt1}
\lim_{n\to\infty}\|(I-cV)^nf\|&=0
\\
\text{and}\quad\lim_{n\to\infty}\|(I+cV)^{-n}f\|&=0.\label{volt2}
\end{align}
\end{lemma}

\begin{proof} First, take the function ${\bf 1}$ being identically 1.
One can easily verify that $(I-cV)^n{\bf1}(t)=L_n(ct)$, where $L_n$
are the Laguerre polynomials:
$$
L_n(t)=\sum_{k=0}^n \frac{n!(-t)^k}{(n-k)!(k!)^2}.
$$
From the well-known asymptotic formulae \cite[Chapter~8]{sz} for
Laguerre polynomials it follows that for any $0<a<b<\infty$,
$\|L_n\|_{L_\infty[0,b]}=O(1)$ and
$\|L_n\|_{L_\infty[a,b]}=O(n^{-1/2})$ as $n\to\infty$. Since
$(I-cV)^n{\bf1}(t)=L_n(ct)$, we immediately obtain $(I-cV)^n{\bf
1}\to 0$ in $L_2[0,1]$ as $n\to\infty$. Now, let $p(t)=\sum_{j=0}^n
p_jt^j$ be any polynomial. Then $p=q(V){\bf 1}$, where
$q(t)=\sum_{j=0}^n j!p_jt^j$. Therefore
$$
(I-cV)^np=(I-cV)^nq(V){\bf1}=q(V)(I-cV)^n{\bf1}\to 0\quad\text{as
$n\to\infty$}
$$
since $(I-cV)^n{\bf1}\to 0$ and $q(V)$ is a bounded linear operator
on $L_2[0,1]$. Since the space of polynomials is dense in
$L_2[0,1]$, we see that (\ref{volt1}) is satisfied for $f$ from a
dense set.

Next, $V+V^*=P$, where $P$ is the orthoprojection onto the
one-dimensional space of constant functions. Hence
$$
\Re \langle Vf,f\rangle=\frac12\langle
(V+V^*)f,f\rangle=\frac12\langle Pf,f\rangle\geq 0\quad\text{for any
$f\in L_2[0,1]$}.
$$
Thus, for any $f\in L_2[0,1]$ with $\|f\|=1$, we have
$$
\|(I+cV)f\|\geq \Re\langle f+cVf,f\rangle\geq \langle f,f\rangle =1.
$$
Hence $(I+cV)^{-1}$ is a contraction. One can also easily verify
that $I-cV=M_c(I+cV)^{-1}M_c^{-1}$, where $M_cf(t)=e^{-ct}f(t)$. It
follows that $I-cV$ and $(I+cV)^{-1}$ are similar. Therefore, the
operator $I-cV$ is power bounded since it is similar to a
contraction. Since $I-cV$ is power bounded and (\ref{volt1}) is
satisfied for $f$ from a dense set, we obtain that (\ref{volt1}) is
satisfied for all $f\in L_2[0,1]$. Finally similarity of $I-cV$ and
$(I+cV)^{-1}$ implies that (\ref{volt2}) is also satisfied for any
$f\in L_2[0,1]$.
\end{proof}

\subsection{The Kitai Criterion and quasisimilarity\label{s2-3}}

It has been noticed by many authors that for any operator satisfying
a certain cyclicity property (like being cyclic, supercyclic,
hypercyclic, satisfying the Hypercyclicity Criterion or satisfying
the Kitai Criterion), any other operator being in certain
quasisimilarity relation with the first one satisfies the same
property. The following lemma is proved in \cite{MP}.

\begin{lemma}\label{quasi}
Let $X$ and $X_0$ be Banach spaces, $T_0\in L(X_0)$ and $T\in L(X)$.
Suppose also that $T_0$ satisfies the Kitai Criterion and there
exists an injective bounded linear operator $J:X_0\to X$ with dense
range such that $JT_0=TJ$. Then $T$ satisfies the Kitai Criterion.
\end{lemma}

%\begin{proof}Since $T_0$ satisfies the Kitai Criterion, the linear
%space $E_0=\{x\in X_0:T_0^nx\to 0\}$ is dense in $X_0$. Since $J$
%has dense range, $E=J(E_0)$ is a dense linear subspace in $X$. Let
%$u\in E$. Then $u=Jx$, where $x\in E_0$. Then $T^nu=T^nJx=JT_0^nx\to
%0$ since $x\in E_0$ and the operator $J$ is bounded. Thus, $T$
%satisfies the first condition of the Kitai Criterion.

%Since $T_0$ satisfies the Kitai Criterion, there exists a dense
%subset $F_0$ of $X$ and a map $S_0:F_0\to F_0$ such that $T_0S_0x=x$
%and $S_0^nx\to 0$ for each $x\in F_0$. Since $J$ has dense range,
%$F=J(F_0)$ is dense in $X$. Since $J$ is injective, for each $u\in
%F$, there is a unique $x_u\in F_0$ such that $Jx_u=u$. Thus, the
%formula $Su=JS_0x_u$ defines a map from $F$ to $F$. For each $u\in
%F$, we have $TSu=TJS_0x_u=JT_0S_0x_u=Jx_u=u$. Moreover, by
%definition of $S$, for each $x\in F_0$, we have $SJx=JS_0x$. The
%equality $SJ=JS_0$ implies $S^nJ=JS_0^n$ for each $n\in\Z_+$ and
%therefore $S^nu=S^nJx_u=JS_0^nx\to 0$ for any $u\in F$. Thus, $T$
%satisfies the Kitai Criterion. \end{proof}

\subsection{The Kitai Criterion for $I+T$, where $T$ is a backward
weighted shift \label{s2-4}}

\noindent{\bf Remark.} \ Let $w=\{w_n\}_{n\in\N}$ and
$u=\{u_n\}_{n\in\N}$ be two bounded sequence of non-zero numbers in
$\K$ such that $|w_n|=|u_n|$ for each $n\in\N$ and $T_w$, $T_u$ be
backward weighted shifts with weight sequences $w$ and $u$
respectively acting on $X$ being either $\ell_p$ for $1\leq
p<\infty$ or $c_0$. Then the operators $T_w$ and $T_u$ are
isometrically similar. Indeed, consider the sequence
$\{d_n\}_{n\in\Z_+}$ defined as $d_0=1$ and
$d_n=\prod_{j=1}^n\frac{w_j}{u_j}$ for $n\geq1$. Then $|d_n|=1$ for
each $n\in\Z_+$ and therefore the diagonal operator $D\in L(X)$,
which acts on the basic vectors according to the formula
$De_n=d_ne_n$ for $n\in\Z_+$, is an isometric isomorphism of $X$
onto itself. One can easily verify that $T_u=D^{-1}T_w D$. That is,
$T_w$ and $T_u$ are isometrically similar.

\medskip

For a bounded linear operator $S$ on a Banach space $X$ we denote
\begin{equation} \label{es}
{\cal E}(S)=\{x\in X:S^nx\to 0\ \ \text{as $n\to\infty$}\}.
\end{equation}
It follows immediately from Definition~2 that an invertible $S\in
L(X)$ satisfies the Kitai Criterion if and only if both ${\cal
E}(S)$ and ${\cal E}(S^{-1})$ are dense in $X$.

\begin{lemma}\label{l-kit} Let $\lambda\in\K\setminus\{0\}$ and $T$ be
the weighted backward shift on $\ell_1$ with the weight sequence
$w_n=1/n$, $n\in\N$. Then the operator $I+\lambda T$ satisfies the
Kitai criterion.
\end{lemma}

\begin{proof} Let $c=|\lambda|/2$. Since $T$ is quasinilpotent, the
operator $I+\lambda T$ is invertible and therefore it suffices to
prove that ${\cal E}(I+\lambda T)$ and ${\cal E}((I+\lambda
T)^{-1})$ are dense in $\ell_1$. According to the above remark the
operators $\lambda T$, $2cT$ and $-2cT$ are isometrically similar.
Hence $I+\lambda T$ is similar to $I-2cT$ and $(I+\lambda T)^{-1}$
is similar to $(I+2cT)^{-1}$. Thus, it suffices to prove that
\begin{align}
&\text{${\cal E}(I-2cT)$ is dense in $\ell_1$} \label{den1}
\\
\text{and}\quad&\text{${\cal E}((I+2cT)^{-1})$ is dense in
$\ell_1$.}\label{den2}
\end{align}

Consider the bounded linear operator $J_0:\ell_\infty\to L_2[0,1]$
defined by the formula
$$
J_0x(t)=\sum_{n=0}^\infty \frac{x_n(1-t)^n}{2^n}.
$$
Naturally identifying $\ell_\infty$ with $\ell_1^*$ and
$(L_2[0,1])^*$ with $L_2[0,1]$, one can easily see that $J_0$ is
$*$-weakly continuous and therefore $J_0=J^{**}$, where
$J=J_0\bigr|_{c_0}$. Uniqueness theorem for analytic functions
implies that $J_0=J^{**}$ is injective and therefore
$J^*:L_2[0,1]\to \ell_1$ has dense range. Clearly $J(c_0)$ contains
all polynomials and therefore is dense in $L_2[0,1]$. Hence $J^*$ is
injective. A direct calculation shows that the dual of the Volterra
operator $V$ on $L_2[0,1]$ acts according to the formula
$$
V^*f(t)=\int_t^1 f(s)\,ds.
$$
Consider also the weighted forward shift $S\in L(c_0)$,
$Se_n=e_{n+1}/(n+1)$ for $n\in\Z_+$, where $\{e_n\}_{n\in\Z_+}$ is
the canonical basis of $c_0$. We shall verify that
\begin{equation}
\label{sim} 2JS=V^*J.
\end{equation}
Indeed, for any $n\in\Z_+$,
\begin{align*}
(2JSe_n)(t)&=\frac{2(Je_{n+1})(t)}{n+1}=\frac{2(1-t)^{n+1}}{(n+1)2^{n+1}}=
\frac{(1-t)^{n+1}}{(n+1)2^{n}}
\\
\text{and}\quad(V^*Je_n)(t)&=\int_{t}^1
\frac{(1-s)^n}{2^n}\,ds=\frac{(1-t)^{n+1}}{(n+1)2^{n}}.
\end{align*}
Thus, $2JSe_n=V^*Je_n$ for each $n\in\Z_+$ and (\ref{sim}) follows.
Taking the adjoint of the both sides of the equality (\ref{sim}) and
taking into account that $S^*=T$, we obtain $2TJ^*=J^*V$, which
immediately implies
\begin{equation}
\label{sim1}
(I-2cT)J^*=J^*(I-cV)\quad\text{and}\quad(I+2cT)^{-1}J^*=J^*(I+cV)^{-1}.
\end{equation}
Applying Lemma~\ref{volterra}, we see that for any $f\in L_2[0,1]$,
$$
(I-2cT)^n J^*f=J^*(I-cV)^nf\to 0\quad \text{and} \quad (I+2cT)^{-n}
J^*f=J^*(I+cV)^{-n}f\to 0\ \ \text{as $n\to\infty$}.
$$
Hence each of the spaces ${\cal E}(I-2cT)$ and ${\cal
E}((I+2cT)^{-1})$ contains $J^*(L_2[0,1])$, which is dense in
$\ell_1$. Thus, (\ref{den1}) and (\ref{den2}) hold and therefore
$I+\lambda T$ satisfies the Kitai Criterion. \end{proof}

\subsection{Proof of Theorem~\ref{bws1} \label{s2-5}}

Condition (\ref{1/n}) implies that there exists $c>0$ for which the
sequence
$$
d_0=1,\quad d_n=\prod_{j=1}^n\frac{c}{nw_n}\quad\text{for}\ \ n\in\N
$$
is bounded. Let $S$ be the weighted backward shift on $\ell_1$ with
the weight sequence $w_n=1/n$, $n\in\N$, $\{e_n\}_{n\in\Z_+}$ be the
canonical basis in $X$ and $\{e'_n\}_{n\in\Z_+}$ be the canonical
basis in $\ell_1$. Since $\{d_n\}_{n\in\Z_+}$ is bounded there
exists a unique bounded linear operator $J:\ell_1\to X$ such that
$Je'_n=d_ne_n$ for $n\in\Z_+$. Since $d_n\neq 0$ for each
$n\in\Z_+$, the operator $J$ is injective and has dense range. Using
the definition of $d_n$, one can easily verify that $TJ=cJS$ and
therefore $(I+T)J=J(I+cS)$. By Lemma~\ref{l-kit} $I+cS$ satisfies
the Kitai Criterion. From Lemma~\ref{quasi} it follows now that
$I+T$ also satisfies the Kitai Criterion.

\subsection{Proof of Theorem~\ref{kitai} \label{s2-6}}

Let $S$ be the weighted backward shift on $\ell_1$ with the weight
sequence $w_n=1/n$, $n\in\N$. By Lemma~\ref{wshi}, there exist $T\in
L(X)$ and an injective bounded linear operator $J:\ell_1\to X$ with
dense range such that $JS=TJ$. Hence $J(I+S)=(I+T)J$. By
Lemma~\ref{l-kit} $I+S$ satisfies the Kitai Criterion. Now
Lemma~\ref{quasi} implies that $I+T$ also satisfies the Kitai
Criterion. The proof is complete.

\section{Concluding remarks \label{s3}}\rm

The operator, we construct in the proof of Theorem~\ref{kitai}, is
the sum of the identity operator and a compact quasinilpotent
operator. In particular, its spectrum is the one-point set $\{1\}$.
It is worth noting that there are Banach spaces, where one can not
expect anything else. Namely, if $X$ is a hereditarily
indecomposable Banach space \cite{HI}, then the spectrum of any
hypercyclic operator $T$ on $X$ is a one-point set $\{\lambda\}$
with $|\lambda|=1$, see, for instance, \cite{gri}.

The sufficient condition of frequent hypercyclicity \cite{frhy}, an
interesting concept recently introduced by Bayart and Grivaux, is
related to a stronger form of the Kitai Criterion. Let us say that a
bounded linear operator $T$ on a Banach space $X$ satisfies the {\it
strong Kitai Criterion} if the spaces $E$ and $F$ in Definition~2
may be chosen to be the same space: $E=F$. Clearly, an invertible
operator satisfies the strong Kitai Criterion if and only if
$\E(T)\cap \E(T^{-1})$ is dense in $X$. We shall immediately see
that this condition is much more restrictive than the Kitai
Criterion.

\begin{lemma} \label{qn} Let $T$ be a bounded linear operator on a
complex Banach space $X$ and $\sigma(T)=\{1\}$. Then $\E(T)\cap
\E(T^{-1})=\{0\}$.
\end{lemma}

\begin{proof} Let $Y$ be the space $\E(T)\cap
\E(T^{-1})$ endowed with the norm
$\ssub{\|y\|}{Y}=\max\{\|T^ny\|:n\in\Z\}$. It is straightforward to
see that $Y$ is a Banach space, $T(Y)\subseteq Y$ and the
restriction $A=T\bigr|_{Y}$, considered as an operator acting on the
Banach space $Y$ is an invertible isometry. Now, if
$z\in\C\setminus\{1\}$, then, using the fact that $T-I$ is
quasinilpotent, we see that $(T-zI)^{-1}=(1-z)^{-1}\sum_{k=0}^\infty
(1-z)^k (T-I)^k$, where the series in the right-hand side is
operator norm absolutely convergent. Note that if $S$ is a bounded
linear operator on $X$ such that $ST=TS$, then $S(Y)\subseteq Y$ and
the restriction of $S$ to $Y$, considered as an operator on $Y$, has
the operator norm not exceeding $\|S\|$. Hence, the series
$(1-z)^{-1}\sum_{k=0}^\infty (1-z)^k (A-I)^k$ is operator norm
absolutely convergent to $(A-zI)^{-1}$. Thus, $z\notin\sigma(A)$.
Hence, $\sigma(A)=\{1\}$. Since $A$ is an isometry, we have that
$\{\|A\|^n\}_{n\in\Z}$ is bounded. A classical theorem due to
Gel'fand asserts that an invertible element $x$ of a unital Banach
algebra is the identity if $\sigma(x)=\{1\}$ and
$\{\|x^n\|\}_{n\in\Z}$ is bounded. Thus, $A=I$ and therefore $Ty=y$
for any $y\in Y$. From the definition of $Y$ it follows that $Ty=y$
for $y\in Y$ happens if and only if $y=0$. Thus, $Y=\{0\}$.
\end{proof}

\begin{proposition}\label{skc}
Let $X$ be a hereditarily indecomposable Banach space. Then there is
no bounded linear operators on $X$ satisfying the strong Kitai
Criterion. \end{proposition}

\begin{proof} The case $\K=\R$
reduces to the case $\K=\C$ by passing to the complexification.
Thus, we can assume that $\K=\C$. Let $S$ be a bounded linear
operator on $X$, satisfying the strong Kitai Criterion. As we have
already mentioned, the fact that $X$ is hereditarily indecomposable
implies that $\sigma(S)$ is a one-point set $\{\lambda\}$ with
$|\lambda|=1$. Then $T=\lambda^{-1}S$ satisfies the strong Kitai
Criterion and $\sigma(T)=\{1\}$. Since $T$ is invertible and
satisfies the strong Kitai Criterion, we have that $\E(T)\cap
\E(T^{-1})$ is dense in $X$, which is not possible according to
Lemma~\ref{qn}. \end{proof}

It is also worth noting that Theorem~\ref{bws1} is surprisingly
sharp. The following observation is due to Atzmon, see
\cite{atz1,atz2}.

\begin{thma}Let $k\in\N$ and $T$ be a bounded linear operator on a Banach
space $X$ such that $\|T^n\|^{1/n}=o(1/n)$ as $n\to\infty$. Then for
$x\in X$, $\|(I+T)^nx\|=O(n^k)$ as $n\to\infty$ if and only if
$T^kx=0$.
\end{thma}

From Theorem~A it immediately follows that ${\cal E}(I+T)=\{0\}$ if
$\|T^n\|^{1/n}=o(1/n)$ as $n\to\infty$. Indeed, if $x\in {\cal
E}(I+T)$ then by Theorem~A we have $Tx=0$ and therefore $(I+T)^nx=x$
for each $n\in\Z_+$. Taking into account that ${\cal E}(S)$ is dense
for each operator $S$ satisfying the Kitai Criterion, we have the
following corollary.

\begin{corollary}\label{antikit} Let $T$ be a bounded linear
operator on a Banach space $X$ such that $\|T^n\|^{1/n}=o(1/n)$ as
$n\to\infty$. Then $I+T$ does not satisfy the Kitai Criterion.
\end{corollary}

From Corollary~\ref{antikit} it follows, in particular, that if $T$
is a weighted backward shift on $\ell_p$ for $1\leq p<\infty$ or on
$c_0$ with the weight sequence $\{w_n\}$ satisfying $w_n=o(n^{-1})$
as $n\to\infty$, then $I+T$ does not satisfy the Kitai Criterion.

It is well-known and easy to see that the spectrum of a backward
weighted shift $T$ is always the disk $\{z\in\C:|z|\leq r\}$, where
$r\geq 0$ is the spectral radius of $T$. As it easily follows from
the results of Chan and Shapiro \cite{chan}, $I+T$ satisfies the
Kitai Criterion if $r>0$. From this point of view Theorem~\ref{bws1}
gives a much more subtle sufficient condition for such operators to
satisfy the Kitai Criterion. In particular, it shows that there are
many quasinilpotent backward weighted shifts $T$ for which $I+T$
satisfies the Kitai Criterion.

Finally observe that the Kitai Criterion makes sense for operators
on arbitrary topological vector spaces. A natural question arises in
relation with Theorem~\ref{kitai}. Namely, does there exist a
continuous linear operator satisfying the Kitai Criterion on any
separable infinite dimensional Fr\'echet space?

\bigskip

{\bf Acknowledgements.} \ The author is grateful to Professors
Sophie Grivaux and Charles Read for their interest and helpful
comments.

\small\rm

\vskip1truecm

\scshape

\noindent Stanislav Shkarin

\noindent Queen's University Belfast

\noindent Department of Pure Mathematics

\noindent University road, BT7 1NN \ Belfast, UK

\noindent E-mail address: \qquad {\tt s.shkarin@qub.ac.uk}


\begin{thebibliography}{99}

\itemsep=-2pt

\bibitem{ansa}S.~Ansari, \it Hypercyclic and cyclic vectors, \rm J. Funct.
Anal. \bf128\rm\  (1995), 374--383

\bibitem{ansa1}S.~Ansari, \it Existence of hypercyclic operators on
topological vector spaces, \rm J. Funct. Anal. \bf148\rm\ (1997),
384--390

\bibitem{HI}S.~Argyros and A.~Tolias, \it Methods in the theory of
hereditarily indecomposable Banach spaces, \rm Mem. Amer. Math. Soc.
\bf 170 \rm (2004), 1--114

\bibitem{atz1}A.~Atzmon, \it Operators which are annihilated by
analytic functions and invariant subspaces, \rm Acta Math.
\bf144\rm\ (1980), 27--63

\bibitem{atz2}A.~Atzmon, \it On the existence of invariant subspaces, \rm J.
Operator Theory \bf11\rm\ (1984), 3--40

\bibitem{frhy}F.~Bayart and S.~Grivaux, \it Frequently hypercyclic
operators, \rm Trans. Amer. Math. Soc. \bf 358\rm\ (2006),
5083--5117

\bibitem{bernal}L.~Bernal-Gonz\'alez, \it On hypercyclic operators on
Banach spaces, \rm Proc. Amer. Math. Soc. \bf127\rm\ (1999),
1003--1010

\bibitem{bp}J.~B\`es and A.~Peris, \it Hereditarily hypercyclic
operators, \rm J. Funct. Anal.  \bf167\rm\  (1999),  94--112

\bibitem{chan}K.~Chan and J.~Shapiro, \it The cyclic behavior of
translation operators on Hilbert spaces of entire functions, \rm
Indiana Univ. Math. J. \bf40\rm\ (1991), 1421--1449

\bibitem{CS}G.~Costakis and M.~Sambarino, \it Topologically mixing
operators, \rm Proc. Amer. Math. Soc. \bf132\rm\ (2003), 385--389

\bibitem{Dv}A.~Dvoretzky, \it A theorem on convex bodies and
application to Banach spaces, \rm  Proc. Nat. Acad. Sci. USA
\bf45\rm\ (1959),  223--226

\bibitem{ger}G.~Herzog, \it On linear operators having supercyclic
vectors, \rm Studia Math. \bf103\rm\ (1992), 295--298

\bibitem{gri}S.~Grivaux, \it Hypercyclic operators, mixing operators
and the bounded steps problem, \rm J. Operator Theory \bf54\rm\
(2005), 147--168

\bibitem{ge1}K.~Grosse-Erdmann, \it Universal families and
hypercyclic operators, \rm Bull. Amer. Math. Soc. \bf36\rm\ (1999),
345--381

\bibitem{ge2}K.~Grosse-Erdmann, \it Recent developments in
hypercyclicity, \rm RACSAM Rev. R. Acad. Cienc. Exactas Fis. Nat.
Ser. A Mat. \bf 97\rm\  (2003), 273--286

\bibitem{kitai}C.~Kitai, \it Invariant closed sets for linear
operators, \rm Thesis, University of Toronto, 1982

\bibitem{MP}F.~Mart\'\i nez-Gim\'enez and A.~Peris, \it Chaos for
backward shift operators, \rm Internat. J. Bifur. Chaos Appl. Sci.
Engrg. \bf12\rm\ (2002), 1703--1715

\bibitem{msz}A.~Montes-Rodr\'\i guez, J.~S\'anchez-\'Alvarez and
J.~Zem\'anek, \it Uniform Abel--Kreiss boundedness and the extremal
behaviour of the Volterra operator, \rm Proc. London Math. Soc.
\bf91\rm\ (2005), 761--788

\bibitem{sal}H.~Salas, \it Hypercyclic weighted shifts, \rm  Trans.
Amer. Math. Soc. \bf 347\rm\ (1995), 993--1004

\bibitem{sz}G.~Szeg\"o, \it Orthogonal polynomials, \rm AMS,
Colloq. Publ., 1959

\end{thebibliography}
\end{document}